\documentclass{article}
\usepackage{amscd,amssymb,latexsym,amsmath,amsthm}
\usepackage{epsfig}
\def\e#1\e{\begin{equation}#1\end{equation}}

\def\ea#1\ea{\begin{align}#1\end{align}}
\def\eq#1{{\rm(\ref{#1})}}
\theoremstyle{plain}
\newtheorem{thm}{Theorem}[section]
\newtheorem{prop}[thm]{Proposition}
\newtheorem{lem}[thm]{Lemma}

\theoremstyle{definition}
\newtheorem{dfn}[thm]{Definition}

\newtheorem{rem}[thm]{Remark}
\def\Re{\mathop{\rm Re}}
\def\Im{\mathop{\rm Im}}
\def\Ric{\mathop{\rm Ric}}
\def\Ker{\mathop{\rm Ker}}

\def\U{\mathbin{\rm U}}
\def\u{\mathbin{\mathfrak u}}
\def\Vol{\mathop{\rm Vol}\nolimits}
\def\Exp{\mathop{\rm Exp}\nolimits}

\def\id{\mathop{\rm id}}
\def\ge{\geqslant}
\def\le{\leqslant}

\def\R{{\mathbin{\mathbb R}}}
\def\Z{{\mathbin{\mathbb Z}}}

\def\C{{\mathbin{\mathbb C}}}

\def\al{\alpha}

\def\na{\nabla}
\def\ga{\gamma}
\def\de{\delta}

\def\ep{\epsilon}
\def\ze{\zeta}

\def\om{\omega}
\def\vp{\varphi}
\def\th{\theta}
\def\up{\upsilon}
\def\Ga{\Gamma}
\def\De{\Delta}
\def\Om{\Omega}

\def\Up{\Upsilon}
\def\pd{\partial}
\def\ts{\textstyle}
\def\ds{\displaystyle}
\def\w{\wedge}
\def\iy{\infty}
\def\lt{\ltimes}
\def\ra{\rightarrow}

\def\longra{\longrightarrow}
\def\t{\times}

\def\ci{\circ}
\def\ti{\tilde}
\def\d{{\rm d}}
\def\f{\frac}
\def\Tn{T^{n}_{a_1,\ldots,\, a_n}}
\def\ot{\otimes}

\def\ha{{\ts\frac{1}{2}}}
\def\bs{\boldsymbol}
\def\epqij{e^{\sqrt{-1}(\th _p+\th _q-\th _i-\th _j)}}
\def\epqijm{e^{\sqrt{-1}(\th _p+\th _q-\th _i-\th _j+\th _m)}}
\def\epqijbm{e^{\sqrt{-1}(\th _p+\th _q-\th _i-\th _j-\th _m)}}
\def\eji{e^{\sqrt{-1}(\th_j-\th_i)}}
\def\eij{e^{\sqrt{-1}(\th_i-\th_j)}}
\def\ei{e^{\sqrt{-1}\th_i}}
\def\ej{e^{\sqrt{-1}\th_j}}
\def\eni{e^{-\sqrt{-1}\th_i}}
\def\enj{e^{-\sqrt{-1}\th_j}}
\def\Rjiii{R_{j\bar{i}i\bar{i}}(p)\eji}

\def\Riiiii{R_{i\bar{i}i\bar{i},i}(p)\ei}
\def\Riiiij{R_{i\bar{i}i\bar{i},j}(p)\ej}
\def\Riiiibi{R_{i\bar{i}i\bar{i},\bar{i}}(p)\eni}
\def\Riiiibj{R_{i\bar{i}i\bar{i},\bar{j}}(p)\enj}
\def\md#1{\vert #1 \vert}

\def\nm#1{\Vert #1 \Vert}
\def\bnm#1{\big\Vert #1 \big\Vert}
\def\ms#1{\vert #1 \vert^2}
\def\an#1{\langle #1 \rangle}
\def\pds#1{\frac{\pd}{\pd #1}}
\def\ban#1{\bigl\langle #1 \bigr\rangle}
\begin{document}
\title{The existence of Hamiltonian stationary Lagrangian
tori in K\"ahler manifolds of any dimension}
\author{ Yng-Ing Lee}

\date{}
\maketitle
\leftline{Department of Mathematics and Taida Institute
of Mathematical Sciences,}
\leftline{National Taiwan University,
Taipei 10617, Taiwan}
\leftline{National Center for Theoretical
Sciences, Taipei Office}
\centerline{email: yilee@math.ntu.edu.tw}
\begin{abstract}
Hamiltonian stationary Lagrangians are Lagrangian submanifolds
that are critical points of the volume functional under
Hamiltonian deformations. They can be considered as a
generalization of special Lagrangians or Lagrangian and minimal
submanifolds. In \cite{JLS}, Joyce, Schoen and the author show
that  given any compact rigid Hamiltonian stationary Lagrangian in
$\C^n$, one can always find a family of Hamiltonian stationary
Lagrangians of the same type in any compact symplectic manifolds
with a compatible metric. The advantage of this result is that it
holds in very general classes. But the disadvantage is that we do
not know where  these examples locate and examples in this family
might be far apart. In this paper, we derive a local condition on
K\"ahler manifolds which ensures the existence of one family of
Hamiltonian stationary Lagrangian tori near a point with given
frame satisfying the criterion. Butscher and Corvino posted a
condition in $n=2$ in \cite{BuCo}. But our condition appears to be
different from theirs.  The condition derived in this paper  not
only works for any dimension, but also for the Clifford torus case
which is not covered by their condition.

\end{abstract}

\section{Introduction}
\label{hst1} {\it Hamiltonian stationary\/} (or {\it H-minimal\/})
Lagrangians were defined and studied by Oh \cite{Oh1,Oh2} in a
K\"ahler manifold $(M, g)$. These objects have stationary volume
amongst Hamiltonian equivalent Lagrangians. The Euler--Lagrange
equation for a Hamiltonian stationary Lagrangian $L$ is
$\d^*\al_H=0$, where $H$ is the mean curvature vector on $L$,
$\al_H$ the 1-form on $L$ defined by $\al_H(\cdot)=\om(H,\cdot)$,
and $\d^*$ the Hodge dual of the exterior derivative~$\d$.

 Special
Lagrangians/ Lagrangian and minimal submanifolds are critical
points of the volume functional of all variations, and Hamiltonian
stationary Lagrangians can be considered as their generalizations.
Hamiltonian stationary Lagrangians are related models for
incompressible elasticity theory and are closely related to the
study of special Lagrangians/ Lagrangian and minimal submanifolds.
Although there are no compact special Lagrangians in $\C^n$, there
are compact Hamiltonian stationary Lagrangians in $\C^n$.  Oh
proves in \cite[Th.~IV]{Oh2} that for $a_1,\ldots,a_n>0$, the
torus $T^n_{a_1,\ldots,a_n}$ in $\C^n$ given by \e
T^n_{a_1,\ldots,a_n}=\bigl\{(z_1,\ldots,z_n)\in\C^n:\md{z_j}=a_j,
\; j=1,\ldots,n\bigr\} \label{hst1eq1} \e is a stable, rigid,
Hamiltonian stationary Lagrangian in $\C^n$. A Hamiltonian
stationary Lagrangian is called {\it stable\/} (or {\it
H-stable\/}) if the second variational formula of the volume
functional among Hamiltonian deformations is nonnegative. A
Hamiltonian stationary Lagrangian in $\C^n$ is called {\it
rigid\/} (or {\it H-rigid\/}) if the Jacobi vector fields for
Hamiltonian variations consist only those from the $\U(n)\lt\C^n$
actions on $\C^n$ (see \cite[\S 2.3]{JLS}). Other compact stable,
rigid, Hamiltonian stationary Lagrangians in $\C^n$ are given in
\cite{AmOh}.

H\'{e}lein and Romon found all Hamiltomian stationary Lagrangian
tori in $\C^2$ and $\C P^2$ via a Weierstrass-type representation
\cite{HR1, HR2}. But  there are very few results on the existence
of Hamiltonian stationary Lagrangians in
 general K\"ahler manifolds. In
\cite{JLS}, Joyce, Schoen and the author obtain families of
compact Hamiltonian stationary Lagrangians in every compact
symplectic manifold $(M, \om)$ with a compatible metric $g$. It in
particular contains the case of K\"ahler manifolds. The result is:
\medskip

\noindent{\bf Theorem \cite{JLS}} {\it Suppose that\/ $(M,\om)$ is
a compact symplectic\/ $2n$-manifold, $g$ is a Riemannian metric
on $M$ compatible with\/ $\om,$ and\/ $L$ is a compact,
Hamiltonian rigid, Hamiltonian stationary Lagrangian in\/ $\C^n$.
Then there exist compact, Hamiltonian stationary Lagrangians\/
$L'$ in $M$ which are diffeomorphic to\/ $L,$ such that\/ $L'$ is
contained in a small ball about some point $p\in M,$ and
identifying $M$ near\/ $p$ with\/ $\C^n$ near\/ $0$ in Darboux
geodesic normal coordinates, $L'$ is a small deformation of\/ $t
L$ for small\/ $t>0$. If\/ $L$ is also Hamiltonian stable, we can
take\/ $L'$ to be Hamiltonian stable.}
\medskip

The method used in \cite{JLS} is first to find Darboux coordinates
at each point which also admit a nice expression on the metric,
and then put a scaled compact Hamiltonian stationary Lagrangian
from $\C^n$ in the Darboux coordinates at each point. These
submanifolds are Lagrangian in $(M, \om, g)$, but not Hamiltonian
stationary yet. One then tries to perturb these approximate
examples in Hamiltonian equivalence class to Hamiltonian
stationary. This involves solving a highly nonlinear equation
whose linearized equation has approximate kernels, and thus it
cannot be done in general. In \cite{JLS}, we first solve the
equation perpendicular to the approximate kernels for examples
near any fixed point and then show that the problem of finding
Hamiltonian stationary Lagrangians, which are critical points of
the volume functional on an infinite dimensional space, can be
reduced to finding critical points of a smooth function on a
finite dimensional compact space when the model from $\C^n$ is in
addition Hamiltonian rigid and $M$ is compact. The existence will
follow from the simple fact that every continuous function has
critical points on a compact set.

The advantage of the above argument is that it only requires
compactness and works in very general classes. But the
disadvantage is that we do not know where the Hamiltonian
stationary Lagrangian locates in $(M, \om, g)$. As a consequence,
the examples obtained at each scaled size $t$ may be far apart for
different $t$. In this paper, we take a different approach in the
second step and resolve this deficit when $M$ is a K\"ahler
manifold and $L=\Tn$. More precisely, we show that

\medskip

\noindent{\bf Theorem \ref{hst4thm2}} {\it Suppose that \/
$(M,\om,g)$ is an n-dimensional K\"ahler manifold, and write $U$
for the $\U(n)$ frame bundle of $M$. The subgroup of diagonal
matrices $T^n\subset \U(n)$ acts on $U$. For any given
$a_i>0,\;i=1,\ldots, n$, define $F_{a_1,\ldots,a_n}: U/T^n\ra \R$
by \/
$F_{a_1,\ldots,a_n}(p,[\up])=\sum_{i=1}^{n}a_i^2R_{i\bar{i}i\bar{i}}(p)$,
where $p\in M, \; \up \in \U(n),$ and the holomorphic sectional
curvature $R_{i\bar{i}i\bar{i}}(p)$ is computed w.r.t. the unitary
frame $\up$ at $p$ whose value is clearly independent of the
representative $\up$ of $[\up]$. Assume that $(p_0,[\up_0])\in
U/T^n$ is a non-degenerate critical point of $F_{a_1,\ldots,a_n},$
then for $t$ small there exist a smooth family $(p(t),[\up(t)])
\in U/T^n$ satisfying $(p(0),[\up(0)])=(p_0,[\up_0])$ and a smooth
family of embedded Hamiltomian stationary Lagrangian tori with
radii $(ta_1,\ldots, ta_n)$ center at $p(t)$ which are invariant
under $T^n$ action and posited w.r.t any representative of
$[\up(t)]$. Moreover, the distance between $(p(t),[\up(t)])$ and
$(p_0,[\up_0])$ in $U/T^n$ is bounded by $ct^2$. The family of
embedded Hamiltomian stationary Lagrangian tori  do not intersect
each other when $t$ is small.}

\medskip

The proof of the theorem is along the same line as in \cite{JLS}
with the following differences:
\begin{itemize}
\setlength{\itemsep}{0pt} \setlength{\parsep}{0pt} \item On a
K\"ahler manifold, we obtain Darboux coordinates with  better
expressions on the metric. And when $L=\Tn$, we can compute the
leading terms in related estimates in explicit forms. \item In the
last step, instead of using compactness to show the existence, we
analyze directly the conditions we need to perturb  approximate
examples to Hamiltonian stationary. This is done by deriving
explicit expressions up to some orders in all related estimates.
Because we do not use  compactness condition, the result also hold
for noncompact K\"ahler manifolds.
\end{itemize}

 Our result is an analogue to the constant mean curvature
(CMC) hypersurface case in a Riemannian manifold $M$. Ye in
\cite{Ye} showed that near a non-degenerate critical point $p$ of
the scalar curvature function on $M$, there exist CMC sphere
foliation near $p$ . The problem of finding a corresponding
condition for Hamiltomian stationary Lagrangian tori on a K\"ahler
manifold is proposed by Schoen, and is  the starting point of our
project in this direction including \cite{JLS}. Butscher and
Corvino proposed a different condition in
  \cite{BuCo} for $n=2$, which
  is the non-degenerate critical point
  of the function $G_{a_1,a_2}(p,\up)=a_1^2R_{1\bar 1}^{\C}(p)+a_2^2R_{2\bar 2}^{\C}(p)$
  on  $U/T^2$.  Here $R_{1\bar 1}^{\C}$ and
  $R_{2\bar 2}^{\C}$ are the complex Ricci curvature.  Note that
  $G_{a_1,a_2}$
   will be a multiple
  of the scalar
  curvature when $a_1=a_2$. It is independent
  of the frame, and thus there won't be any non-degenerate critical point
   of $G_{a_1,a_2}$ in this case. In contrast to that,
  our condition not only works
for any dimension, but also cover the Clifford torus (i.e., with
same radii) case. Because the dimension does not match, the family
of Hamiltomian stationary Lagrangian tori with radii
$(ta_1,\ldots, ta_n)$ won't form a foliation.

This paper is organized as follows. In \S\ref{hst2} we give basic
definitions and derive new Darboux coordinates which will play an
essential role in the paper.  Some important and involving
estimates are given in \S3. Section 4 consist of the set up for
the perturbation and the proof of the main theorem. A different
proof which is more close to our approach in \cite{JLS} is
presented in the last section. This more geometrical simple proof
also give another justification for the  computation in \S3. The
first proof has its own interests which demonstrate the general
scheme of the perturbation method. So we  present both proofs in
the paper.
\medskip

\noindent {\bf Acknowledgements:} I  benefit a lot from the joint
project  with Joyce and Schoen in \cite{JLS}, and the set up of
this paper very much follows that in \cite{JLS}. I would like to
express my special gratitude to both of them.  Brendle's comments
in my talk in Columbia university remind me  to revisit one of my
earlier approaches and leads to the different proof in the last
section. I am indebted to him for his enlightening comments, and
also to Joyce for his many helpful comments in a earlier version
of the paper.
\section{Notation and Darboux coordinates}
\label{hst2}
\subsection{Lagrangian and Hamiltonian stationary} \label{hst21}
 We will assume $(M, \om,g)$ to be K\"ahler through this paper, and refer to
 \S 2 in \cite{JLS} for more
 detailed discussions on the background material.
\begin{dfn}  A
submanifold $L$ in $(M, \om)$ is called {\it Lagrangian} if $\dim
L=n=\ha\dim M$ and $\om\vert_L\equiv 0$. It follows that the image
of the tangent bundle $TL$ under the complex structure $J$ is
equal to the normal bundle $T^{\perp}L$. \label{hst2def1}
\end{dfn}

Let $F:M\ra\R$ be a smooth function on $M$. The {\it Hamiltonian
vector field\/} $v_F$ of $F$ is the unique vector field satisfying
$v_F\cdot\om=\d F$. The Lie derivative satisfies ${\cal
L}_{v_F}\om=v_F\cdot\d\om +\d(v_F\cdot\om)=0$, so the trajectory
of $v_F$ gives a 1-parameter family of diffeomorphisms  $\Exp
(sv_F):M\ra M$ for $s\in\R$  which preserve $\om$. It is called
the {\it Hamiltonian flow\/} of $F$. If $L$ is a compact
Lagrangian in $M$ then $\Exp(sv_F)L$ is also a compact Lagrangian
in~$M$.

\begin{dfn} A compact Lagrangian submanifold $L$ in $(M, \om,g)$ is called
{\it Hamiltonian stationary}, or {\it H-minimal}, if it is a
critical point of the volume functional among Hamiltonian
deformations. That is, $L$ is Hamiltonian stationary if \e
\ts\frac{\d}{\d s}\Vol_g\bigl(\Exp(sv_F)L\bigr)\big\vert_{s=0}=0
\label{hst21eq2} \e for all smooth $F:M\ra\R$. By Oh
\cite[Th.~I]{Oh2}, \eq{hst21eq2} is equivalent to the
Euler--Lagrange equation \e \d^*\al_H=0, \label{hst21eq3} \e where
$H$ is the mean curvature vector of $L$, and $\al_H=
(H\cdot\om)\vert_L$ is the associated 1-form of $H$ on $L$, and
$\d^*$ is the Hodge dual of the exterior derivative $\d$ on $L$,
computed using the metric $h=g\vert_L$. \label{hst2def2}
\end{dfn}

When $(M,\om, g)$ is a Calabi-Yau manifold, one can choose a
holomorphic $(n,0)$-form $\Om$ on $M$ with $\nabla\Om=0$,
normalized so that
\begin{equation*}
\om^n/n!=(-1)^{n(n-1)/2}(i/2)^n\Om\w\bar\Om.
\end{equation*}
If $L$ is an oriented Lagrangian in $M$, then $\Om\vert_L\equiv
e^{i\th}\d V_{L}$, where $\d V_{L}$ is the induced volume form
from the metric $g$. It defines {\it Lagrangian angle\/} $\th: L
\ra \R/2\pi\Z$ on $L$. The submanifold $L$ is called {\it special
Lagrangian\/} if $\th$ is constant. On a Hamiltonian stationary
Lagrangian, the Lagrangian angle $\th$ is harmonic. If moreover,
the image of $\th$ lies in $\R$ (and here $L$ is compact),  then
the Hamiltonian stationary Lagrangian $L$ is indeed special
Lagrangian since every harmonic function on a compact manifold
must be constant.

At a  Hamiltonian stationary Lagrangian $L$, one can compute to
find that \e \ts\frac{\d^2}{\d
s^2}\Vol_g\bigl(\Exp(sv_F)L\bigr)\big\vert_{s=0} =\ban{{\cal
L}f,f}_{L^2(L)}, \label{hst21eq5} \e where $f=F\vert_L$ and \e
{\cal L}f= \De^2 f+\d^*\al_{\Ric^{\perp}(J\,\na f)}-2\d^*
\al_{B(JH,\na f)} -JH(JH(f)). \label{hst21eq6} \e In
\eq{hst21eq6}, $\De f=\d^*\d f$ is the positive Laplacian on $L$,
$B(\cdot,\cdot)$ is the second fundamental form on $L,$ and
$\Ric^{\perp}(v)$ is a normal vector field along $L$ defined by
$\Ric(v,w)=\an{\Ric^{\perp}(v),w}$ for any normal vector $w$. Note
also that by the Lagrangian condition $JH$ is tangent to $L$.

Given a smooth function $f$ on a Lagrangian $L$, we can extend it
to a smooth function $F$ on $M$ and consider $L_s=\Exp (sv_F)L$
whose mean curvature vector is denoted by $H_s$. One can derive
the linearized operator of $-\d^*\al_{H_f}=-\d^*\al_{H_1}$ and
obtain that \e -\f{\d}{\d s}\,(\d^*\al_{H_s})\big\vert_{s=0}={\cal
L}f. \label{hst21eq7}\e Here $L$ does not need to be Hamiltonian
stationary.

 Oh proves in \cite[Th.~IV]{Oh2} that the torus $\Tn$ in $\C^n$ given
by \eq{hst1eq1} is Hamiltonian stationary with \eq{hst21eq5}
nonnegative definite (Hamiltonian stable), and $\Ker{\cal L}$ at
$\Tn$ consist of functions of the following form
\begin{equation*}
Q(z_1,\ldots,z_n)=\ts a+\sum_{j=1}^n(b_jz_j+\bar{b}_j\bar{z}_j)
+\sum_{j \ne k}^nc_{jk}z_j\bar{z}_k
\end{equation*}
restricted on $\Tn$, where $ a \in \R, \, b_j,\, c_{jk}\in \C,$
and\/ $c_{jk}=\bar c_{kj}$ (Hamiltonian rigid, see \cite{JLS}). If
we write $\Tn$ in polar coordinates \e
\Tn=\bigl\{(a_1e^{\sqrt{-1}\th_1},\ldots,a_ne^{\sqrt{-1}\th_n})\in\C^n:\th_i
\in [0,2\pi), i=1,\ldots,n\bigr\} \label{hst21eq8}, \e then $
\Ker{\cal L}$ is spanned by \e 1,\;\cos\th_i,\; \sin\th_i,\;
\cos(\th_i-\th_j),\; \sin(\th_i-\th_j)\label{hst21eq9} \e  for
\/$i,\, j=1,\ldots,n$ and \/$i\ne j$ \cite{Oh2}.

\subsection{New Darboux coordinates}
\label{hst22}

The convention  for curvature operator used in this paper is
$$R(X,Y)Z=\na _{X}\na _Y Z-\na _{Y}\na _X Z -\na _{[X,Y]} Z,$$
and
$$R_{ijkl}=\an{R(\pds{x^i}, \pds{x^j})\pds{x^k}, \pds{x^l}}.$$
We use the same notion for complex curvature operator and denote
$$R_{i\bar{j}k\bar{l}}=\an{R(\pds{z^i}, \pds{\bar{z}^j})\pds{z^k}, \pds{\bar{z}^l}}.$$
 The basic definitions and properties for
curvature of  K\"ahler metrics can be found in \cite{Tian}.

Denote $\C^n$ with complex coordinates $(z_1,\ldots,z_n)$, where
$z_j=x_j+\sqrt{-1}y_j$.  Define the standard Euclidean metric
$g_0$, K\"ahler form $\om_0$, and complex structure $J_0$ on
$\C^n$ by \begin{equation*}
\begin{split}
g_0&=\ts\sum_{j=1}^n\ms{\d z_j}=\sum_{j=1}^n(\d x_j^2+\d y_j^2),\\
\om_0&=\ts\frac{\sqrt{-1}}{2}\sum_{j=1}^n\d
z_j\w\d\bar{z}_j=\sum_{j=1}^n\d
x_j\w\d y_j,\quad\text{and}\\
J_0&=\ts\sum_{j=1}^n\bigl(\sqrt{-1}\d z_j \ot \frac{\pd}{\pd
z_j}-\sqrt{-1}\d\bar{z}_j\ot \frac{\pd}{\pd\bar{z}_j}\bigr)
=\sum_{j=1}^n\bigl(\d x_j\ot\frac{\pd}{\pd y_j}- \d
y_j\ot\frac{\pd}{\pd x_j}\bigr),
\end{split}
\end{equation*}
 noting that $\d z_j=\d
x_j+\sqrt{-1}\d y_j$ and $\frac{\pd}{\pd
z_j}=\ha\bigl(\frac{\pd}{\pd x_j}-\sqrt{-1}\frac{\pd}{\pd
y_j}\bigr)$.

Darboux's Theorem says that we can find local coordinates near any
point on a symplectic manifold such that the symplectic structure
is like $\om_0$ in $\C^n$ in these coordinates, which will be
called Darboux coordinates. Because we need a good control on the
metric as well, we will redo the argument to find better  Darboux
coordinates.
  We first start with holomorphic normal
 coordinates at points in a K\"ahler manifold, and proceed as in \cite[Prop.~3.2]{JLS}
 to convert them into Darboux coordinates. To meet our need,
  we will
 not only  derive the leading
 coefficients of the metric in this new Darboux coordinates, but also
 the coefficients of the next order. More precisely, we
have
\begin{prop} Let\/ $(M,\om, g)$ be a compact  n-dimensional K\"ahler manifold
with associate K\"ahler form $\om$ and let\/ $U$ be the $\U(n)$
frame bundle of\/ $M$. Then for small\/ $\ep>0$ we can choose a
family of embeddings $\Up_{p,\up}:B_\ep\ra M$ depending smoothly
on $(p,\up)\in U,$ where $B_\ep$ is the ball of radius $\ep$
about\/ $0$ in $\C^n,$ such that for all\/ $(p,\up)\in U$ we have:
\begin{itemize}
\setlength{\itemsep}{0pt} \setlength{\parsep}{0pt} \item[{\rm(i)}]
$\Up_{p,\up}(0)=p$ and\/ $\d\Up_{p,\up}\vert_0= \up:\C^n\ra T_pM;$
\item[{\rm(ii)}] $\Up_{p,\up\ci\ga}\equiv\Up_{p,\up}\ci\ga$ for
all\/ $\ga\in\U(n);$ \item[{\rm(iii)}]
$\Up_{p,\up}^*(\om)=\om_0=\ts\frac{\sqrt{-1}}{2} \sum_{j=1}^n\d
z_j\w\d\bar{z}_j;$ and \item[{\rm(iv)}] $\Up_{p,\up}^*(g)=\ts
g_0+\frac{1}{2} \sum\Re\bigl(R_{i\bar{j}k\bar{l}}(p)z^i z^k \d\bar
z^{j}\d\bar z^{l}\bigr)+\frac{1}{5}
\sum\Re\bigl(R_{i\bar{j}k\bar{l},m}(p)z^i z^k z^m\d\bar
z^{j}\d\bar z^{l}\bigr)+\frac{2}{5}\sum
\Re\bigl(R_{i\bar{j}k\bar{l},\bar{m}}(p)z^i z^k \bar z^m\d\bar
z^{j}\d\bar z^{l}\bigr)+O\bigl(\md{\bs z}^4\bigr),$ where
$g_0=\ts\sum_{j=1}^n\ms{\d z_j}$.

\end{itemize}
 \label{hst2prop1}
\end{prop}
\begin{proof}
Given $(p,\up)\in U,$ we can find holomorphic coordinates that is
an embedding $\Up_{p,\up}':B_{\ep'}\ra M$ satisfying (i),(ii), and
\ea
(\Up_{p,\up}')^*(g)=&g_0-\sum_{i,j,k,l}R_{i\bar{j}k\bar{l}}(p)z^k
\bar{z}^l\d z^{i}\d\bar{z}^{j}+O(\md{\bs
z}^4)\nonumber\\&\quad-\f{1}{2}\sum_{i,j,k,l,m}\bigl(R_{i\bar{j}k\bar{l},m}(p)z^k
\bar{z}^lz^m+R_{i\bar{j}k\bar{l},\bar{m}}(p)z^k
\bar{z}^l\bar{z}^m\bigr)\d z^{i}\d\bar{z}^{j}. \nonumber
 \ea
 The pull back K\"{a}hler form has similar corresponding expression, and
 $\Up_{p,\up}'$ is smooth in $p,\up$.
 As in the proof of
 \cite[Prop.~3.2]{JLS}, we can use Moser's method \cite{Mose}
 for proving Darboux' Theorem  to modify the maps
 $\Up'_{p,\up}$ to $\Up_{p,\up}$ with $\Up_{p,\up}^* (\om)=\om_0$.
 Define closed 2-forms $\om^s_{p,\up}$ on $B_{\ep'}$
for $(p,\up)\in U$ and $s\in[0,1]$ by
$\om^s_{p,\up}=(1-s)\om_0+s(\Up_{p,\up}')^*(\om)$. Then there
exist a family of 1-forms $\ze_{p,\up}$ on $B_{\ep'}$ satisfying
$\f{\sqrt{-1}}{2}\,\d\ze_{p,\up}=\om_0-(\Up_{p,\up}')^*(\om)$,
which can be taken as  \ea
\ze_{p,\up}=&\f{1}{4}R_{i\bar{j}k\bar{l}}(p)z^k
\bar{z}^l\bigl(-\bar{z}^j\d
z^{i}+z^i\d\bar{z}^{j}\bigr)+\f{1}{10}R_{i\bar{j}k\bar{l},m}(p)z^k
\bar{z}^l z^m\bigl(-\bar{z}^j\d
z^{i}+z^i\d\bar{z}^{j}\bigr)\nonumber\\&+\f{1}{10}R_{i\bar{j}k\bar{l},\bar{m}}(p)z^k
\bar{z}^l \bar{z}^m\bigl(-\bar{z}^j\d
z^{i}+z^i\d\bar{z}^{j}\bigr)+O(\md{\bs z}^5) .\nonumber \ea
 We use
the convention that repeated indices stand for a summation
whenever there is no confusion. In the following we compute the
first term of $\d\ze_{p,\up}$
 as an example to demonstrate the argument, \ea
&\d \bigl(\f{1}{4}R_{i\bar{j}k\bar{l}}(p)z^k
\bar{z}^l\bigl(-\bar{z}^j\d
z^{i}+z^i\d\bar{z}^{j}\bigr)\bigr)\nonumber\\=&-\f{1}{4}R_{i\bar{j}k\bar{l}}(p)\bar{z}^l\bar{z}^j\d
z^k \w\d z^{i}-\f{1}{4}R_{i\bar{j}k\bar{l}}(p)z^k \bar{z}^j \d
\bar{z}^l\w\d z^{i} -\f{1}{4}R_{i\bar{j}k\bar{l}}(p)z^k \bar{z}^l
\d \bar{z}^j\w \d z^{i} \nonumber \\& +
\f{1}{4}R_{i\bar{j}k\bar{l}}(p)\bar{z}^l z^{i}\d z^k \w\d
\bar{z}^j+\f{1}{4}R_{i\bar{j}k\bar{l}}(p)z^k z^{i} \d
\bar{z}^l\w\d \bar{z}^j +\f{1}{4}R_{i\bar{j}k\bar{l}}(p)z^k
\bar{z}^l  \d z^{i} \w \d
\bar{z}^j\nonumber\\=&+\f{1}{4}R_{i\bar{l}k\bar{j}}(p)z^k
\bar{z}^j \d z^{i} \w \d
\bar{z}^l+\f{1}{4}R_{i\bar{j}k\bar{l}}(p)z^k \bar{z}^l  \d z^{i}
\w \d \bar{z}^j\nonumber\\&+\f{1}{4}R_{k\bar{j}i\bar{l}}(p)z^i
\bar{z}^l  \d z^{k} \w \d
\bar{z}^j+\f{1}{4}R_{i\bar{j}k\bar{l}}(p)z^k \bar{z}^l  \d z^{i}
\w \d \bar{z}^j\nonumber\\=&R_{i\bar{j}k\bar{l}}(p)z^k \bar{z}^l
\d z^{i} \w \d \bar{z}^j \nonumber.\ea In the second equality we
use
$R_{i\bar{j}k\bar{l}}(p)=R_{k\bar{j}i\bar{l}}(p)=R_{i\bar{l}k\bar{j}}(p)$
which is implied by the K\"{a}hler condition, and the last
equality follows by changing the indices. The other terms can be
computed similarly, noting that there are two $\bar{z}$ with $\d
z^i$ and three $z$ with $\d \bar{z}^j$ which make the coefficient
from $\f{1}{10}$ to $\f{1}{2}$.

Now let $v^s_{p,\up}$ be the unique vector field on $B_{\ep'}$
with $v^s_{p,\up}\cdot\om^s_{p,\up}=\f{\sqrt{-1}}{2}\ze_{p,\up}$.
If we denote $v^s_{p,\up}=2\Re
\sum_{j}a^{s,j}_{p,\up}\pds{z^j}=\sum_{j}\Re
(a^{s,j}_{p,\up})\pds{x^j}+\Im (a^{s,j}_{p,\up})\pds{y^j}$, the
coefficient $a^{s,j}_{p,\up}$ will be \ea  \f{1}{4}
R_{i\bar{j}k\bar{l}}(p)z^k \bar{z}^lz^i+\f{1}{10}
R_{i\bar{j}k\bar{l},m}(p)z^k \bar{z}^l
z^mz^i+\f{1}{10}R_{i\bar{j}k\bar{l},\bar{m}}(p)z^k \bar{z}^l
\bar{z}^mz^i+O(\md{\bs z}^5) . \nonumber\ea  For $0<\ep\le\ep'$ we
construct a family of embeddings $\vp^s_{p,\up}:B_\ep\ra B_{\ep'}$
with $\vp^0_{p,\up}=\id:B_\ep\ra B_\ep\subset B_{\ep'}$ by solving
the system $\frac{\d}{\d
s}\vp_{p,\up}^s=v_{p,\up}^s\ci\vp_{p,\up}^s$. By compactness of
$[0,1]\t U$, this is possible provided $\ep>0$ is small enough.
Then $(\vp_{p,\up}^s)^*(\om^s_{p,\up})=\om_0$ for all $s$, so that
$(\vp_{p,\up}^1)^*\bigl((\Up_{p,\up}')^*(\om)\bigr) =\om_0$. The
$j$-th component of  $\vp_{p,\up}^1$ in $z$ coordinates is
 \ea
z^j+\f{1}{4} R_{i\bar{j}k\bar{l}}(p)z^k
\bar{z}^lz^i+\f{1}{10}R_{i\bar{j}k\bar{l},m}(p)z^k \bar{z}^l
z^mz^i+\f{1}{10}R_{i\bar{j}k\bar{l},\bar{m}}(p)z^k \bar{z}^l
\bar{z}^mz^i+O(\md{\bs z}^5) . \label{hst2eq0}\ea Define
$\Up_{p,\up}=\Up_{p,\up}'\ci\vp_{p,\up}^1$. Then $\Up_{p,\up}$
depends smoothly on $p,\up$. Direct computations give \ea
\Up_{p,\up}^*(g) = &g_0+\frac{1}{2}
\sum_{i,j,k,l}\Re\bigl(R_{i\bar{j}k\bar{l}}(p)z^i z^k \d\bar
z^{j}\d\bar z^{l}\bigr)\nonumber \\&\quad+\frac{1}{5}
\sum_{i,j,k,l,m}\Re\bigl(R_{i\bar{j}k\bar{l},m}(p)z^i z^k
z^m\d\bar z^{j}\d\bar z^{l}\bigr)\nonumber \\&\quad +\frac{2}{5}
\sum_{i,j,k,l,m}\Re\bigl(R_{i\bar{j}k\bar{l},\bar{m}}(p)z^i z^k
\bar z^m\d\bar z^{j}\d\bar z^{l}\bigr)+O\bigl(\md{\bs z}^4\bigr).
\label{hst2eq1}\ea The different coefficients $\f{1}{5}$ and
$\f{2}{5}$ in \eq{hst2eq1} comes from the fact that their
corresponding terms in \eq{hst2eq0} respectively have one
$\bar{z}$ and two $\bar{z}$. The rest of the proof is the same as
\cite[Prop.~3.2]{JLS}, and we refer to the proof there for
details.
\end{proof}
\begin{rem}
The K\"{a}hler manifold $M$ does not need to be compact if we
allow $\ep$ depending on points.
\end{rem}

\section{Approximate examples with estimates}
\label{hst3}

For $0<t\le R^{-1}\ep$, consider the dilation map $t:B_R\ra B_\ep$
mapping $t:(z_1,\ldots,z_n)\mapsto (tz_1,\ldots,tz_n)$. Then
$\Up_{p,\up}\ci t$ is an embedding $B_R\ra M$, so we can consider
the pullbacks $(\Up_{p,\up}\ci t)^*(\om)$ and $(\Up_{p,\up}\ci
t)^*(g)$. Define a Riemannian metric $g^t_{p,\up}$ on $B_R$ by
$g^t_{p,\up}= t^{-2}(\Up_{p,\up}\ci t)^*(g)$. It depends smoothly
on $t\in(0,R^{-1}\ep]$ and $(p,\up)\in U$, and satisfies \ea
g^t_{p,\up} = &g_0+\frac{t^2}{2}
\sum_{i,j,k,l}\Re\bigl(R_{i\bar{j}k\bar{l}}(p)z^i z^k \d\bar
z^{j}\d\bar z^{l}\bigr)\nonumber \\&\quad+\frac{t^3}{5}
\sum_{i,j,k,l,m}\Re\bigl(R_{i\bar{j}k\bar{l},m}(p)z^i z^k
z^m\d\bar z^{j}\d\bar z^{l}\bigr)\nonumber \\&\quad
+\frac{2t^3}{5}
\sum_{i,j,k,l,m}\Re\bigl(R_{i\bar{j}k\bar{l},\bar{m}}(p)z^i z^k
\bar z^m\d\bar z^{j}\d\bar z^{l}\bigr)+O\bigl(t^4\md{\bs
z}^4\bigr). \label{hst3eq1}\ea
 Since $t^{-2}(\Up_{p,\up}\ci t)^*(g)$ is
compatible with $t^{-2}(\Up_{p,\up}\ci t)^*(\om)$, we have that
$g^t_{p,\up}$ is compatible with the fixed symplectic form $\om_0$
on $B_R$ for all $t,p,\up$. Moreover, there are {\it uniform
estimates\/} on these metrics, which are summarized in the
following proposition.

\begin{prop} There exist positive constants
$C_0,C_1,C_2,\ldots$ such that for all\/ $t\in(0,\ha R^{-1}\ep]$
and\/ $(p,\up)\in U,$ the metric
$g^t_{p,\up}=t^{-2}(\Up_{p,\up}\ci t)^*(g)$ on\/ $B_R$ satisfies
the estimates \e \nm{g^t_{p,\up}-g_0}_{C^0}\le
C_0t^2\quad\text{and}\quad \nm{\pd^kg^t_{p,\up}}_{C^0}\le
C_kt^{k+1} \quad\text{for $k=1,2,\ldots,$} \label{hs3eq3} \e where
norms are taken w.r.t.\ $g_0,$ and\/ $\pd$ is the Levi-Civita
connection of\/~$g_0$. \label{hs3prop2}
\end{prop}
\begin{proof}
This is the same as \cite[Prop. 3.4]{JLS}. But since we have a
better estimate on the metric from Proposition~\ref{hst2prop1}, we
can increase the order on $t$ by 1. \end{proof}

 We can assume $\sum_{j=1}^{n}a_j^2=1$ for simplicity.
 The image $(\Up_{p,\up}\ci t ) (\Tn)$ is a Lagrangian contained
in a $B_{2t}$ ball at $p$ in $M$. Since the geometry  of
$B_{2t}(p)$ in $(M,\om, g)$ is the same as $(B_2, \om_0,
g^t_{p,\up})$ in $\C^n$, we will do all the computations and
discussions in $(B_2, \om_0, g^t_{p,\up})$ instead for simplicity.
In the coordinates $z^j=r_je^{\sqrt{-1} \th _j},\; j=1,\ldots,
n$,\/
 the metric $g^t_{p,\up}$ becomes \ea g^t_{p,\up}=&\sum (\d
r_i^{2}+r_i^2 \d\th_i^2)+\sum (t^2 \Re A_{ij}+t^3 \Re C_{ij})(\d
r_i \d r_j-r_ir_j\d\th_i\d\th_j)\nonumber\\&  +\sum (t^2 \Im
A_{ij}+t^3 \Im C_{ij})( r_i \d \th_i\d r_j+r_j\d
r_i\d\th_j)+O\bigl(t^4\md{\bs z}^4\bigr),\label{hst2eq2} \ea where
\ea A_{ij}=A_{ji} =&\f{1}{2} \sum_{p,q}
R_{p\bar{i}q\bar{j}}(p)r\!_pr\!_q \epqij, \nonumber\\
C_{ij}=C_{ji} = &\f{1}{5}\sum_{p,q,m}
R_{p\bar{i}q\bar{j},m}(p)r\!_pr\!_qr\!_m
\epqijm\nonumber\\&+\f{2}{5}\sum_{p,q,m} R_{p\bar{i}q\bar{j},\bar
m}(p)r\!_pr\!_qr\!_m \epqijbm. \label{hst2eq3}\ea The restriction
of $g^t_{p,\up}$ on $\Tn $ is \e h^t_{p,\up}=\sum a_i^2\d
\th_i^{2}-\sum a_ia_j(t^2\Re A_{ij}+t^3\Re
C_{ij})\d\th_i\d\th_j+O(t^4),\label{hst2eq4} \e where $\md{\bs z}$
has been assumed to be $1$ on $\Tn$. For simplicity, we omit the
restriction of $A_{ij}$ and $C_{ij}$ on $\Tn$ in \eq{hst2eq4}, and
will denote $g^t_{p,\up}$ by $g_t$ and $h^t_{p,\up}$ by $h_t$
when there is no confusion. A direct computation yields \ea
&h_t^{ij}=\f{1}{a_i^2}\de _{ij}+ \f{t^2}{a_ia_j}\Re A_{ij}+
\f{t^3}{a_ia_j}\Re C_{ij}+O(t^4),\nonumber\\& g_t^{r_i r_j}=\de
_{ij}-t^2\Re A_{ij}-t^3\Re C_{ij}+O\bigl(t^4\md{\bs
z}^4\bigr),\nonumber\\& g_t^{\th_i \th_j}=\f{1}{r_i^2}\de _{ij}+
\f{t^2}{r_ir_j}\Re A_{ij}+ \f{t^3}{r_ir_j}\Re
C_{ij}+O\bigl(t^4\md{\bs z}^4\bigr),\nonumber\\&g_t^{r_i
\th_j}=-\f{t^2}{r_j}\Im A_{ij}-\f{t^3}{r_j}\Im
C_{ij}+O\bigl(t^4\md{\bs z}^4\bigr).\label{hst2eq5}\ea Now we are
ready to compute the associate $\d^*\al_t$ of the initial
Lagrangian $\Tn \subset (B_2, \om_0, g^t_{p,\up})$, and estimate
how far it is from being Hamiltonian stationary.
\begin{lem}
Denote the mean curvature vector on $\Tn$ with respect to $g_t$ by
$H_t$ and let $\al _t=H_t \cdot \om_0=\sum \al _t^k \d \th_k$.
Then
 \ea \d^*\al_t=& -
\sum_{i,j}\bigl(\f{1}{a_ia_j}\f{\pd ^2 \Im (t^2A_{ij}+t^3
C_{ij})}{\pd \th_i \pd \th _j}+\f{1}{2a_i}\f{\pd ^2 \Re
(t^2A_{jj}+t^3 C_{jj})}{\pd \th_i \pd r_i}\nonumber\\&\quad+
\f{1}{2a_i^2}\f{\pd \Re (t^2A_{jj}+t^3 C_{jj})}{\pd \th
_i}\bigr)+O(t^4),\label{hst2eq6}\ea where $A_{ij}$ and $C_{ij}$
are as defined in \eq{hst2eq3} \label{hst2lem1}\end{lem}

 \begin{proof} Because $\om_0=\sum r_k\d r_k\w \d\th _k$ and $\Tn$
 is Lagrangian, it follows that
$\ds \al_t^k=a_k\sum_{i,j}h_{t}^{ij}(\bar{\Ga}_t)_{\th_i
\th_j}^{r_k},$ where $(\bar{\Ga}_t)_{ab}^{c}$ is the Christoffel
symbol for the metric $g_t$. A direct computation gives \ea
\al_t^k=&-1-\Re (t^2 A_{kk}+t^3 C_{kk})+\sum_i
\f{a_k}{a_i}\bigl(\Re (t^2 A_{ik}+t^3 C_{ik})+\f{\pd \Im (t^2
A_{ik}+t^3 C_{ik})}{\pd \th
_i}\bigr)\nonumber\\&\quad+\f{a_k}{2}\sum_{i}\f{\pd r_i^2 \Re (t^2
A_{ii}+t^3 C_{ii})}{a_i^2 \pd r_k}+O(t^4).\nonumber\ea
Further
computation shows that the $t^2$ and $t^3$ terms of $\al_t^k$ are
\ea B_k=&\sum_{i }\f{a_k}{a_i}\f{\pd \Im (t^2 A_{ik}+t^3
C_{ik})}{\pd \th_i}+\f{a_k}{2}\sum_i\f{\pd \Re (t^2 A_{ii}+t^3
C_{ii})}{\pd r _k}\nonumber\\&+\sum_{i}\f{a_k}{a_i}\Re (t^2
A_{ik}+t^3 C_{ik}).\label{hst2eq7}\ea Recall that
$$ \d^*\al_t=-\sum \frac{\pd h_t^{ij}}{\pd
\th_i}\,\al_t^j-\sum h_t^{ij}\,\frac{\pd \al_t^j}{\pd
\th_i}-\f{1}{2}\sum
 h_t^{ij}\al_t^j\frac{\pd}{\pd
\th_i}\bigl(\ln\det((h_t)_{kl})\bigr). $$ Therefore, \ea
\d^*\al_t=&\sum_{i,j}\f{1}{a_ia_j} \f{\pd \Re (t^2A_{ij}+t^3
C_{ij})}{\pd \th _i}-\sum_{i}\f{1}{a_i^2}\f{\pd B_i}{\pd
\th_i}\nonumber\\&-\sum _{i,j}\f{1}{2a_i^2}\f{\pd \Re
(t^2A_{jj}+t^3 C_{jj})}{\pd \th _i}+O(t^4).\label{hst2eq8}\ea From
\eq{hst2eq7}, we have \ea \sum_{i}\f{1}{a_i^2}\f{\pd B_i}{\pd
\th_i}=&\sum_{i,j}\f{1}{a_ia_j}\bigl(\f{\pd ^2 \Im (t^2A_{ij}+t^3
C_{ij})}{\pd \th_i \pd \th _j}+\f{\pd \Re (t^2A_{ij}+t^3
C_{ij})}{\pd\th _i}\bigr)\nonumber\\&\quad+ \f{1}{2a_i}\f{\pd ^2
\Re (t^2A_{jj}+t^3 C_{jj})}{\pd \th_i \pd r_i} \label{hst2eq9} \ea
Combining \eq{hst2eq8} and \eq{hst2eq9}, we get \ea \d^*\al_t=& -
\sum_{i,j}\bigl(\f{1}{a_ia_j}\f{\pd ^2 \Im (t^2A_{ij}+t^3
C_{ij})}{\pd \th_i \pd \th _j}+\f{1}{2a_i}\f{\pd ^2 \Re
(t^2A_{jj}+t^3 C_{jj})}{\pd \th_i \pd r_i}\nonumber\\&\quad+
\f{1}{2a_i^2}\f{\pd \Re (t^2A_{jj}+t^3 C_{jj})}{\pd \th
_i}\bigr)+O(t^4).\nonumber\ea
\end{proof}

We now compute the orthogonal projection of  $\d^*\al_t$ to
$\Ker{\cal L}$. Recall that  $\Ker{\cal L}$ of $\Tn$ in $\C^n$ is
spanned by $1, \, \cos \th_i, \sin \th_i,  \cos (\th_i-\th_j),
\sin (\th_i-\th_j)$ for \/$i,\, j=1,\ldots,n$ and \/$i\ne j$. From
 \eq{hst2eq3}, it follows that in the leading terms only $A_{ij}$
 can project to  $\cos (\th_i-\th_j)$ and $\sin (\th_i-\th_j)$,
 while only $C_{ij}$ can project to $\cos\th_i$ and $ \sin \th_i$.
 The result is summarized in the following
lemma.
\begin{lem}
Denote the orthogonal projection from $L^2(\Tn)$ w.r.t. $g_0$ to
$\Ker{\cal L}$ by $P$. Then \ea P\, \d^*\al_t=&2t^2\sum_{j>i}\Im
\bigl(-a_j^2R_{j\bar{i}j\bar{j}}(p) +a_i^2
R_{j\bar{i}i\bar{i}}(p)\bigr)\f{\cos{(\th_j-\th_i)}}{a_i
a_j}\nonumber\\&+2t^2\sum_{j>i}\Re
\bigl(-a_j^2R_{j\bar{i}j\bar{j}}(p)
+a_i^2R_{j\bar{i}i\bar{i}}(p)\bigr)\f{\sin{(\th_j-\th_i)}}{a_i
a_j} \nonumber\\&+t^3\sum_{ij} \bigl(\Im a_i^2
R_{i\bar{i}i\bar{i},j}(p)\f{\cos \th_j}{a_j}+\Re a_i^2
R_{i\bar{i}i\bar{i},j}(p)\f{\sin \th_j}{a_j}\bigr)+O(t^4)
\label{hst2eq10}.\ea
 \label{hst2lem2}
\end{lem}
\begin{proof} From the expression of $A_{ij}$ in \eq{hst2eq3} and
$\Ker{\cal L},$ we have \ea
PA_{ii}=&\f{1}{2}a_i^2R_{i\bar{i}i\bar{i}}(p)+\sum_{j\ne i} a_i
a_j \Rjiii, \quad \text{and for } i \ne j,\nonumber
\\ P A_{ij}=&\sum_{q\ne i,j} a_i a_q
R_{i\bar{i}q\bar{j}}(p)e^{\sqrt{-1}(\th_q-\th_j)}+\sum_{q\ne i,j}
a_j a_q
R_{q\bar{i}j\bar{j}}(p)e^{\sqrt{-1}(\th_q-\th_i)}\nonumber\\&+a_i
a_j
R_{i\bar{i}j\bar{j}}(p)+\f{1}{2}a_i^2R_{i\bar{i}i\bar{j}}(p)\eij
+\f{1}{2}a_j^2R_{j\bar{i}j\bar{j}}(p)\eji. \nonumber \ea
Therefore, \ea -P\,\sum_{i,j}\f{1}{a_ia_j}\f{\pd ^2 \Im
A_{ij}}{\pd \th_i \pd \th _j}=&-P\sum_i\f{1}{a_i^2} \f{\pd ^2 \Im
A_{ii}}{\pd \th_i ^2}-P\,\sum_{i \ne j}\f{1}{a_ia_j}\f{\pd ^2 \Im
A_{ij}}{\pd \th_i \pd \th _j} \nonumber\\=&\sum_{j \ne
i}\f{a_j}{a_i}\Im
\bigl(R_{j\bar{i}i\bar{i}}(p)-R_{j\bar{i}j\bar{j}}(p)\bigr)\eji.
\label{hst2eq11}\ea
 Similar computation gives \e
-P\,\sum_{i,j}\f{1}{2a_i}\f{\pd ^2 \Re A_{jj}}{\pd \th_i \pd
r_i}=\f{1}{2}\sum_{i\ne j}\bigl(\f{a_i}{a_j}-\f{a_j}{a_i}\bigr)\Im
\Rjiii,\label{hst2eq12}\e and \e-P\,\sum_{i,j}\f{1}{2a_i^2}\f{\pd
\Re A_{jj}}{\pd \th _i}=\f{1}{2}\sum_{i\ne
j}\bigl(\f{a_i}{a_j}-\f{a_j}{a_i}\bigr)\Im
\Rjiii.\label{hst2eq13}\e Combining \eq{hst2eq11}, \eq{hst2eq12},
and \eq{hst2eq13},  we obtain that the $t^2$ coefficient of
$\d^*\al_t$ is \ea &\sum_{i\ne j}\Im\bigl(
\bigl(-\f{a_j}{a_i}R_{j\bar{i}j\bar{j}}(p)
+\f{a_i}{a_j}R_{j\bar{i}i\bar{i}}(p)\bigr)\eji\bigr)
\nonumber\\=&2\sum_{j>i}\Im \bigl(-a_j^2R_{j\bar{i}j\bar{j}}(p)
+a_i^2 R_{j\bar{i}i\bar{i}}(p)\bigr)\f{\cos{(\th_j-\th_i)}}{a_i
a_j}\nonumber\\&+2\sum_{j>i}\Re
\bigl(-a_j^2R_{j\bar{i}j\bar{j}}(p)
+a_i^2R_{j\bar{i}i\bar{i}}(p)\bigr)\f{\sin{(\th_j-\th_i)}}{a_i
a_j}. \label{hst2eq14}\ea

We also have \ea PC_{ii}=&\f{1}{5}a_i^3 \Riiiii+\f{3}{5}\sum_{j\ne
i}a_i^2 a_j \Riiiij, \nonumber\\& +\f{2}{5}a_i^3
\Riiiibi+\f{2}{5}\sum_{j\ne i}a_i^2 a_j \Riiiibj,
\quad\text{and}\nonumber\\
PC_{ij}=&\f{3}{5}a_i^2 a_j
R_{i\bar{i}j\bar{j},i}(p)\ei+\f{3}{5}a_i a_j^2
R_{i\bar{i}j\bar{j},j}(p)\ej\nonumber\\&+\f{6}{5}\sum_{q \ne
i,j}a_i a_j
a_qR_{i\bar{i}j\bar{j},q}(p)e^{\sqrt{-1}\th_q}+\f{4}{5}\sum_{q}a_i
a_j a_q
R_{i\bar{i}j\bar{j},\bar{q}}(p)e^{-\sqrt{-1}\th_q}\nonumber\ea for
$ i \ne j$. Further computation gives \ea
-P\,\sum_{i,j}\f{1}{a_ia_j}\f{\pd ^2 \Im C_{ij}}{\pd \th_i \pd \th
_j}=&-P\sum_i\f{1}{a_i^2} \f{\pd ^2 \Im C_{ii}}{\pd \th_i
^2}-P\,\sum_{i \ne j}\f{1}{a_ia_j}\f{\pd ^2 \Im C_{ij}}{\pd \th_i
\pd \th _j} \nonumber\\=&\sum_i \f{a_i}{5} \Im
R_{i\bar{i}i\bar{i},i}(p)\ei +\f{2a_i}{5}\Im \Riiiibi
\nonumber\\=&-\sum_i \f{a_i}{5} \Im \Riiiii. \label{hst2eq15}\ea
We use $\Im \Riiiibi=-\Im \Riiiii$ in the last equality. For the
other terms, we similarly have \ea
&-P\,\sum_{i,j}\f{1}{2a_i}\f{\pd ^2 \Re C_{jj}}{\pd \th_i \pd
r_i}\nonumber\\=&-\sum_{i}\f{1}{10a_i}\f{\pd^2 }{\pd \th_i \pd
r_i}\bigl(r_i^3 \Re \Riiiii+2r_i^3 \Re
\Riiiibi\bigr)\nonumber\\&-\sum_{i\ne j}\f{1}{10a_j}\f{\pd^2 }{\pd
\th_j \pd r_j}\bigl(3r_i^2 r_j \Re \Riiiij+2r_i^2r_j\Re \Riiiibj
\bigr)\nonumber\\=&\f{9}{10}\sum_{i}a_i \Im
\Riiiii+\f{1}{2}\sum_{i\ne j}\f{a_i ^2}{a_j} \Im
\Riiiij,\label{hst2eq16}\ea and \ea
&-P\,\sum_{i,j}\f{1}{2a_i^2}\f{\pd \Re C_{jj}}{\pd \th
_i}\nonumber \\=&-\sum_{i}\f{3a_i}{10}\f{\pd \Re \Riiiii  }{\pd
\th_i}-\sum_{i\ne j}\f{a_i^2}{2a_j}\f{\pd \Re \Riiiij}{\pd \th_j }
\nonumber\\=&\f{3}{10}\sum_{i}a_i \Im \Riiiii+\f{1}{2}\sum_{i\ne
j}\f{a_i ^2}{a_j} \Im \Riiiij.\label{hst2eq17}\ea Putting
\eq{hst2eq15}, \eq{hst2eq16} and  \eq{hst2eq17} together, we
conclude that the $t^3$ coefficient of $\d^*\al_t$ is \ea
&\sum_{i}a_i \Im \Riiiii+\sum_{i\ne j}\f{a_i ^2}{a_j} \Im \Riiiij
\nonumber\\=&\sum_{ij} \Im a_i^2 R_{i\bar{i}i\bar{i},j}(p)\f{\cos
\th_j}{a_j}+\Re a_i^2 R_{i\bar{i}i\bar{i},j}(p)\f{\sin
\th_j}{a_j}.\label{hst2eq18} \ea Thus \eq{hst2eq14} and
\eq{hst2eq18} yield  \eq{hst2eq10}.
\end{proof}
\section{Perturbation}
\label{hst4} We will formulate a family of fourth-order nonlinear
elliptic partial differential operators $P^t_{p,\up}:C^\iy(\Tn)\ra
C^\iy(\Tn)$ depending on $(p,\up)\in U$ and small $t>0$, such that
$C^1$-small $f\in C^\iy(\Tn)$ correspond to Lagrangians
$L_{p,\up}^{t,f}$ in $M$, and $L_{p,\up}^{t,f}$ is Hamiltonian
stationary when $P^t_{p,\up}(f)=0$.

We first set up the problem and introduce some related properties.
Let $L$ be a real $n$-manifold. Then its cotangent bundle $T^*L$
has a {\it canonical symplectic form} $\hat\om$, defined as
follows. Let $(x_1,\ldots,x_n)$ be local coordinates on $L$.
Extend them to local coordinates $(x_1,\ldots,x_n,y_1,\ldots,y_n)$
on $T^*L$ such that $(x_1,\ldots,y_n)$ represents the 1-form
$y_1\d x_1+\cdots+y_n \d x_n$ in $T_{(x_1,\ldots,x_n)}^*L$. Then
$\hat\om=\d x_1\w\d y_1+ \cdots+\d x_n\w\d y_n$. Identify $L$ with
the zero section in $T^*L$. Then $L$ is a {\it Lagrangian
submanifold\/} of $(T^*L,\hat\om)$. The following theorem
 shows that any compact Lagrangian submanifold
$L$ in a symplectic manifold looks locally like the zero section
in~$T^*L$.

\medskip
\noindent{\bf Lagrangian Neighbourhood Theorem
\cite[Th.~3.33]{McSa}} {\it Let\/ $(M,\om)$ be a symplectic
manifold and\/ $L\subset M$ a compact Lagrangian submanifold. Then
there exists an open tubular neighbourhood\/ $T$ of the zero
section $L$ in $T^*L,$ and an embedding $\Phi:T\ra M$ with\/
$\Phi\vert_L=\id:L\ra L$ and\/ $\Phi^*(\om)=\hat\om,$ where
$\hat\om$ is the canonical symplectic structure on~$T^*L$.}
\medskip

We shall call $T,\Phi$ a {\it Lagrangian neighbourhood\/} of $L$.
Such neighbourhoods are useful for parametrizing nearby Lagrangian
submanifolds of $M$. Suppose that $\ti L$ is a Lagrangian
submanifold of $M$ which is $C^1$-close to $L$. Then $\ti L$ lies
in $\Phi(T)$, and is the image $\Phi(\Ga_\al)$ of the graph
$\Ga_\al$ of a unique $C^1$-small 1-form $\al$ on $L$. As $\ti L$
is Lagrangian and $\Phi^*(\om)=\hat\om$ we see that
$\hat\om\vert_{\Ga_\al}\equiv 0$. But $\hat\om\vert_{\Ga_\al}
=-\pi^*(\d\al)$, where $\pi:\Ga_\al\ra L$ is the natural
projection. Hence $\d\al=0$, and $\al$ is a {\it closed\/
$1$-form}. This establishes a 1-1 correspondence between
$C^1$-small closed 1-forms on $L$ and Lagrangian submanifolds $\ti
L$ close to $L$ in~$M$.

Making $T$ smaller if necessary, we can suppose $T$ is of the form
\e T=\bigl\{(p,\al):\text{$p\in L$, $\al\in T_p^* L$,
$\md{\al}<\de$}\bigr\} \label{hs4eq1} \e for some small $\de>0$,
where $\md{\al}$ is computed using the metric $g_0\vert_L$.  Now
take $L=\Tn\subset \C^n$. Let $\Tn \subset\Phi(T)\subset
B_2\subset\C^n$ and $f\in C^\iy(\Tn)$ with $\nm{\d f}_{C^0}<\de$.
Define the graph $\Ga_{\d f}$ of $\d f$  to be $\Ga_{\d
f}=\bigl\{(q,\d f\vert_q):q\in \Tn \bigr\}$. Then $\Ga_{\d f}$ is
an embedded Lagrangian submanifold  in $(T,\hat\om)$. Since
$\Phi^*(\om_0)=\hat\om$, we see that $\Phi(\Ga_{\d f})$ is
Lagrangian in $(B_2,\om_0)$.

Let $0<t\le\f{1}{2}\ep$. For each $f\in C^\iy(\Tn)$ with $\nm{\d
f}_{C^0}<\de$, define $L_{p,\up}^{t,f}=\Up_{p,\up}\ci t\ci \Phi
(\Ga_{\d f})$. Then $L_{p,\up}^{t,f}$ is an embedded submanifold
of $M$ diffeomorphic to $\Tn$, and as $\Phi(\Ga_{\d f})$ is
Lagrangian in $(B_2,\om_0)$ and $(\Up_{p,\up}\ci
t)^*(\om)=t^2\om_0$, we see that $L_{p,\up}^{t,f}$ is Lagrangian
in $(M,\om)$. We can further restrict $\int_{\Tn}f\,\d V_0=0$
because $f$ and $f+c$ define the same Lagrangian submanifold. Here
$\d V_0$ is the induced volume form on $\Tn$ w.r.t. $g_0$. Denote
the induced metric of $g^t_{p,\up}$ on $\Phi(\Ga_{\d f})$ by
$h^{t,f}_{p,\up}$ and $\Phi_f:q\in \Tn \mapsto\Phi(q,\d
f\vert_q).$  We define
\e P^t_{p,\up}:\bigl\{f\in
C^\iy(\Tn):\nm{\d f}_{C^0}<\de \bigr\}\longra
C^\iy(\Tn)\label{hs4eq8} \e
 to be the
Euler--Lagrangian operator of \ea
F^t_{p,\up}(f)&=\Vol_{g^t_{p,\up}}\Phi (\Ga_{\d f})=\int_{\bs
\Phi(\Ga_{\d f})}\d V_{h^{t,f}_{p,\up}}\nonumber\\&=\int_{
\Tn}(\Phi_f)^*\bigl(\d
V_{h^{t,f}_{p,\up}}\bigr)\nonumber\\&=\int_{
\Tn}G^t_{p,\up}\bigl(q,\d f\vert_q,\nabla\d f\vert_q\bigr)\,\d
V_0, \label{hs4eq9} \ea where $\nabla$ is the Levi-Civita
connection of the induced metric $h_0$ of $g_0$ on $\Tn$. Assume
that $ -\d^*\al_{H_{f}}$ is computed w.r.t  $g^t_{p,\up}$ at $\Phi
(\Ga_{\d f})$, then \e P^t_{p,\up}(f)=-\bigl(\sqrt{\det
\bigl(\Phi_f^*(h^{t,f}_{p,\up})\bigr)}/\sqrt{\det(h_0)}\bigr)\,\d^*\al_{H_{f}}
=-J^{t,f}_{p,\up}\,\d^*\al_{H_{f}}.\label{hslt4eq1}\e Here we
identify the function $\d^*\al_{H_{f}}$ on $\Phi (\Ga_{\d f})$
with its pull back  on $\Tn$ for simplicity. Because
$J^{t,f}_{p,\up}\ne 0$, it implies that $P^t_{p,\up}(f)\equiv 0$
if and only if $\d^*\al_{H_{f}}\equiv 0$. Thus a zero of
$P^t_{p,\up}$ will give a Hamiltonian stationary Lagrangian w.r.t.
$g^t_{p,\up}$. We choose $P^t_{p,\up}(f)$ instead of
$\d^*\al_{H_{f}}$  just for technical reasons.

Noting that $g^0_{p,\up}=g_0$, we define $F_0(f)=\Vol_{g_0}\Phi
(\Ga_{\d f})$. Denote the corresponding operator at $t=0$ by
$P_0$, and the linearized operators of $P^t_{p,\up}$ and $P_0$ at
$0$ by ${\cal L}^t_{p,\up}$ and $\cal L$ respectively. Here $\cal
L$ is the same as \eq{hst21eq6} w.r.t. $g_0$, but ${\cal
L}^t_{p,\up}$ is slightly different from the one in \eq{hst21eq6}
w.r.t. $g^t_{p,\up}$. We then have
\begin{prop} {\bf \cite[Prop. 4.1]{JLS}}
Let any\/ $k\ge 0,$ $\ga\in(0,1),$ and small\/ $\de'>0$ and
$\ze>0$ be given. Then if\/ $t>0$ is sufficiently small, for all\/
$f\in C^{k+4,\ga}(\Tn)$ with\/ $\nm{\d f}_{C^0}\le\ha\de$ and
$\nm{\nabla \d f}_{C^0}\le \de',$ and all $(p,\up)\in U$ we have
 \e
\bnm{P^t_{p,\up}(f)-P_0(f)}{}_{C^{k,\ga}}\le\ze
\quad\text{and\/}\quad \bnm{{\cal L}^t_{p,\up}(f)-{\cal
L}(f)}{}_{C^{k,\ga}}\le \ze \nm{f}_{C^{k+4,\ga}}. \label{hs4eq15}
\e That is, by taking $t$ small we can suppose $P^t_{p,\up}$ and
${\cal L}^t_{p,\up}$ are arbitrarily close to $P_0$ and $\cal L$
as operators from $C^{k+4,\ga}(\Tn)$ to $C^{k,\ga}(\Tn)$ (on their
respective domains) uniformly in\/~$(p,\up)\in U$. \label{hs4prop}
\end{prop}
\begin{rem} From Proposition~\ref{hs3prop2},
 it follows that we only need to take $ct^2\le \ze$
for some fixed constant $c$ (also see the Appendix). \end{rem}
\begin{thm} {\bf \cite[Th. 5.1]{JLS}} Suppose $0<t\le\ha
\ep$ is sufficiently small and fixed. Then for all\/ $(p,\up)\in
U,$ there exists $f^t_{p,\up}\in C^\iy(\Tn)$ satisfying \e
P^t_{p,\up}(f^t_{p,\up})\in\Ker{\cal L} \quad\text{and\/}\quad
f^t_{p,\up}\perp \Ker{\cal L}, \label{hs5eq1} \e where
$f^t_{p,\up}\perp \Ker{\cal L}$ means $f^t_{p,\up}$ is
$L^2$-orthogonal to $\Ker{\cal L}$. Furthermore $f^t_{p,\up}$ is
the unique solution of \eq{hs5eq1} with
$\nm{f^t_{p,\up}}_{C^{4,\ga}}$ small, and\/ $f^t_{p,\up}$ depends
smoothly on~$(p,\up)\in U$. \label{hs5thm}
\end{thm}
Because $\Tn$ is $T^n$ invariant, it induces a $T^n$ action on the
cotangent bundle of $\Tn$, and $T$ is $T^n$-invariant. Moreover,
we can  choose $\Phi$  to be equivariant under the actions of
$T^n$ on $T$ and $\C^n$ following the proof of the
dilation-equivariant Lagrangian Neighbourhood Theorem
in~\cite[Th.~4.3]{Joyc2}. Furthermore, the functions $f^t_{p,\up}$
in Theorem \ref{hs5thm} satisfy $f^t_{p,\up\ci\ga}\equiv
f^t_{p,\up}\ci\ga$ for $\ga\in T^n$. Define
$L_{p,\up}^t=L_{p,\up}^{t,\smash{f^{\smash{t}}_{\smash{p,\up}}}}$
for $(p,\up)\in U$ and a smooth map $H^t:U\ra\Ker{\cal L}$ by
$H^t:(p,\up)\mapsto P^t_{p,\up} (f^t_{p,\up})$. The map $H^t$ is
$T^n$-equivariant, and depends on $t$ smoothly as $t$ changes. We
refer to \cite{JLS} for a detailed discussion on the setting and
properties.

Now we are ready to state and prove our main result:
\begin{thm}  Suppose that \/
$(M,\om,g)$ is an n-dimensional K\"ahler manifold, and write $U$
for the $\U(n)$ frame bundle of $M$. The subgroup of diagonal
matrices $T^n\subset \U(n)$ acts on $U$. For any given
$a_i>0,\;i=1,\ldots, n$, define $F_{a_1,\ldots,a_n}: U/T^n\ra \R$
by \/
$F_{a_1,\ldots,a_n}(p,[\up])=\sum_{i=1}^{n}a_i^2R_{i\bar{i}i\bar{i}}(p)$,
where $p\in M, \; \up \in \U(n),$ and the holomorphic sectional
curvature $R_{i\bar{i}i\bar{i}}(p)$ is computed w.r.t. the unitary
frame $\up$ at $p$ whose value is clearly independent of the
representative $\up$ of $[\up]$. Assume that $(p_0,[\up_0])\in
U/T^n$ is a non-degenerate critical point of $F_{a_1,\ldots,a_n},$
then for $t$ small there exist a smooth family $(p(t),[\up(t)])
\in U/T^n$ satisfying $(p(0),[\up(0)])=(p_0,[\up_0])$ and a smooth
family of embedded Hamiltomian stationary Lagrangian tori with
radii $(ta_1,\ldots, ta_n)$ center at $p(t)$ which are invariant
under $T^n$ action and posited w.r.t any representative of
$[\up(t)]$. Moreover, the distance between $(p(t),[\up(t)])$ and
$(p_0,[\up_0])$ in $U/T^n$ is bounded by $ct^2$. The family of
embedded Hamiltomian stationary Lagrangian tori  do not intersect
each other when $t$ is small. \label{hst4thm2}
\end{thm}
\begin{proof}
By Theorem~\ref{hs5thm}, the problem of finding Hamiltonian
stationary Lagrangians is reduced to finding zeros of $H^t$.
Because we will change $(p,\up) \in U$, we now rewrite $\al_t$ in
\S 3  as $\al^t_{p,\up}$ to indicate its dependency. We have
 $$P^t_{p,\up} (0)=-\bigl(\sqrt{\det
(h^{t}_{p,\up})}/\sqrt{\det(h_0)}\bigr)
\,\d^*\al^t_{p,\up}=-\d^*\al^t_{p,\up}+O(t^4)$$ from \eq{hst2eq4}
and Lemma~\ref{hst2lem1}.  Thus
 $\bnm{P^t_{p,\up}
(0)}{}_{C^{k,\ga}}\le ct^2$ by Lemma~\ref{hst2lem1} again. The
Implicit Function Theorem used in the proof of
Theorem~\ref{hs5thm} will then   give
$\nm{f^t_{p,\up}}{}_{C^{k+4,\ga}}\le ct^2$. The constant $c$ may
need to be modified at different places, but we still use the same
symbol. It follows from the Appendix that \e P^t_{p,\up}
(f^t_{p,\up})=-P\, \d^*\al^t_{p,\up} +O(t^4),\label{hslteq5}\e
where $P$ is the orthogonal projection from $L^2(\Tn)$ w.r.t.
$g_0$ to $\Ker{\cal L}$. We remark that from the definition of
$P^t_{p,\up} (f)$ in \eq{hslt4eq1} we have
$$\int_{\Tn}P^t_{p,\up} (f^t_{p,\up})\,\d V_0=0.$$ Now we
need to find $(p(t),\up(t))\in U$ such that the coefficients of
$\cos \th_i, \;\sin \th_i,$  $\cos (\th_i-\th_j),$ and $\sin
(\th_i-\th_j)$ for $P^t_{p,\up} (f^t_{p,\up})$   all vanish for
\/$i,\, j=1,\ldots,n$ and \/$i\ne j$. Because $H^t$ is
$T^n$-equivariant, if $(p(t),\up(t))$ is a zero of $H^t$, so is
$(p(t),\up(t) \ci\ga)$ for any $\ga \in T^n$. But they determine
the same Hamiltonian stationary Lagrangian torus since
$f^t_{p,\up\ci\ga}\equiv f^t_{p,\up}\ci\ga$. That is, the
Hamiltonian stationary Lagrangian torus obtained is $T^n$
invariant.

 Suppose that $(p_0,[\up_0])\in U/T^n$ is a critical point of
$F_{a_1,\ldots,a_n}$. We can also consider $F_{a_1,\ldots,a_n}$ as
a function on $U$, and $(p_0,\up_0)\in U$ is its critical point.
One then has $\sum_i a_i^2 R_{i\bar{i}i\bar{i},j}(p_0)=0$ for any
$j$ by applying  variations which vary $p$ in $M$. It follows that
the $t^3$ terms of $H^t(p,\up)=P^t_{p,\up} (f^t_{p,\up})$ vanish
at $(p_0,\up_0)$ by Lemma~\ref{hst2lem2} and \eq{hslteq5}. Suppose
$a_{ij} \in \u (n),\;i< j,$ satisfy
$$a_{ij}\,e_i=e_j,\quad a_{ij}\,e_j=-e_i, \quad \text{and}\quad
a_{ij}\,e_k=0\quad \text{for}\quad k \ne i,j,$$ and $b_{ij} \in \u
(n),\;i< j,$ satisfy
$$b_{ij}\,e_i=-\sqrt{-1}e_j,\quad b_{ij}\,e_j=-\sqrt{-1}e_i, \quad \text{and}\quad
b_{ij}\,e_k=0\quad \text{for}\quad k \ne i,j.$$  Applying a
variation along $a_{ij}\in \u(n)$ at  $(p_0,\up_0)$, it yields \ea
0&=2a_i^2
R_{j\bar{i}i\bar{i}}(p_0)+2a_i^2R_{i\bar{j}i\bar{i}}(p_0)-2a_j^2R_{i\bar{j}j\bar{j}}(p_0)
-2a_j^2R_{j\bar{i}j\bar{j}}(p_0)\nonumber\\&=4\Re \bigl(
a_i^2R_{j\bar{i}i\bar{i}}(p_0)-a_j^2R_{j\bar{i}j\bar{j}}(p_0)\bigr),
\label{hst2eq19}\ea where we have used
$R_{i\bar{j}i\bar{i}}=\overline{R_{j\bar{i}i\bar{i}}}$. Similarly,
applying a variation along $b_{ij}\in \u(n)$ at $(p_0,\up_0)$, it
yields \ea 0&=-2a_i^2
\sqrt{-1}R_{j\bar{i}i\bar{i}}(p_0)+2a_i^2\sqrt{-1}R_{i\bar{j}i\bar{i}}(p_0)
-2a_j^2\sqrt{-1}R_{i\bar{j}j\bar{j}}(p_0)
+2a_j^2\sqrt{-1}R_{j\bar{i}j\bar{j}}(p_0)\nonumber\\&=4\Im \bigl(
a_i^2R_{j\bar{i}i\bar{i}}(p_0)-a_j^2R_{j\bar{i}j\bar{j}}(p_0)\bigr).
 \label{hst2eq20}\ea From Lemma~\ref{hst2lem2} and \eq{hslteq5},
the equalities \eq{hst2eq19} and \eq{hst2eq20} lead to that the
$t^2$ terms of $H^t(p,\up)=P^t_{p,\up} (f^t_{p,\up})$ vanish at
$(p_0,\up_0)$.

  We denote \ea P^t_{p,\up} (f^t_{p,\up})=&\sum_{j>i}D^t_{ij}(p,\up)\f{\cos{(\th_j-\th_i)}}{a_i
a_j}+\sum_{j>i}E^t_{ij}(p,\up)\f{\sin{(\th_j-\th_i)}}{a_i
a_j}\nonumber\\&+\sum_{j}F_j(p,\up)\f{\cos
\th_j}{a_j}+\sum_{j}G_j(p,\up)\f{\sin \th_j}{a_j},\nonumber\ea and
for $t\ne 0$ define a new map $G^t:U\ra \bigl\{f\in\Ker{\cal
L}:\int_{\Tn}f\,\d V_0=0\bigr\}$ by \ea
G^{t}(p,\up)=&-\sum_{j>i}\f{D^t_{ij}(p,\up)}{t^2}\f{\cos{(\th_j-\th_i)}}{a_i
a_j}-\sum_{j>i}\f{E^t_{ij}(p,\up)}{t^2}\f{\sin{(\th_j-\th_i)}}{a_i
a_j}\nonumber\\&+\sum_{j}\f{F_j(p,\up)}{t^3}\f{\cos
\th_j}{a_j}-\sum_{j}\f{G_j(p,\up)}{t^3}\f{\sin
\th_j}{a_j}.\nonumber\ea  From \eq{hslteq5}, Lemma~\ref{hst2lem2}
and the above discussions we have \ea
G^{t}(p,\up)=&\f{1}{2}\sum_{j>i}\bigl(\f{\pd
F_{a_1,\ldots,a_n}(p,\up)}{\pd
b_{ij}}+O(t^2)\bigr)\f{\cos{(\th_j-\th_i)}}{a_i
a_j}\nonumber\\&+\f{1}{2}\sum_{j>i}\bigl(\f{\pd
F_{a_1,\ldots,a_n}(p,\up)}{\pd
a_{ij}}+O(t^2)\bigr)\f{\sin{(\th_j-\th_i)}}{a_i
a_j}\nonumber\\&+\f{1}{2}\sum_{j}\bigl(\f{\pd
F_{a_1,\ldots,a_n}(p,\up)}{\pd y_{j}}+O(t^2)\bigr)\f{\cos
\th_j}{a_j}\nonumber\\&+\f{1}{2}\sum_{j}\bigl(\f{\pd
F_{a_1,\ldots,a_n}(p,\up)}{\pd x_{j}}+O(t^2)\bigr)\f{\sin
\th_j}{a_j}.\label{hst4eq21}\ea  In above we use the observation
that only $t$'s  odd power terms can contribute to the
coefficients of $\cos \th_j$ and $ \sin \th_j$ for $P^t_{p,\up}
(f^t_{p,\up})$ from the derivation of Proposition~\ref{hst2prop1},
Lemma~\ref{hst2lem1} and Lemma~\ref{hst2lem2}.  We can extend
$G^t$ smoothly to $t=0$ in
 \eq{hst4eq21} and consider it as a smooth map $G$ on
$[0,\ep)\t U$. Because $(p_0,[\up_0])$ is a non-degenerate
critical point of $F_{a_1,\ldots,a_n}$, it follows that
$G(0,p_0,\up_0)=0$ and the differential $dG^0\vert_{(p_0,\up_0)}$
is surjective. Noting that the subspace in $\u(n)$ perpendicular
to the Lie algebra of $T^n$ is spanned by $a_{ij}$ and $b_{ij}$,
the Implicit Function Theorem implies that there exists a smooth
function $(p(t),[\up(t)])\in U/T^n$ for small $t$ with
$(p(0),[\up(0)])=(p_0,[\up_0])$, such that $G^t(p(t),\up(t))=0$
for any representative $\up(t)$ of $[\up(t)]$. For $t\ne 0$ the
zeros of $H^t$ and $G^t$ are the same. Hence $H^t(p(t),\up(t))=0$.
Moreover, the distance between $(p(t),[\up(t)])$ and
$(p_0,[\up_0])$ is bounded by $ct^2$ and therefore the Hamiltomian
stationary Lagrangian tori with radii $(ta_1,\ldots, ta_n)$
obtained will not intersect each other for small $t$. This
completes the proof of the theorem.
\end{proof}

\section{A different proof for Theorem~\ref{hst4thm2}}
In \cite{JLS}, we define a smooth function $K^t:U\ra\R$ by \e
K^t(p,\up)=t^{-n}\Vol_g\bigl(
L_{p,\up}^t\bigr)=\Vol_{g^t_{p,\up}}\Phi (\Ga_{\d
f^t_{p,\up}}),\label{hst5eq1}\e and prove that under suitable
identification $\d K^t\vert_{(p,\up)}$ is the same as $H^t(p,\up)$
\cite[Prop. 6.2]{JLS}. Hence finding zeros of $H^t(p,\up)$ is
equivalent to finding critical points of $K^t$. Noting that $K^t$
is also smooth in $t$ and we will do the expansion of $K^t$ in $t$
to analyze the critical points of $K^t$. Recall that
$F^t_{p,\up}(f)=\Vol_{g^t_{p,\up}}\Phi (\Ga_{\d f})$ and
$F_0(f)=\Vol_{g_0}\Phi (\Ga_{\d f})$. We have
\begin{prop}
The function $K^t(p,\up)=F^t_{p,\up}(0)+O(t^4).$ That is, the
leading terms of $K^t(p,\up)$ in $t$ are the same as  the area of
$\Tn$ w.r.t. $g^t_{p,\up}$. \label{hst5prop1}
\end{prop}
\begin{proof}

We pull back the induced metric of $g^t_{p,\up}$ on $\Phi (\Ga_{\d
f^t_{p,\up}})$ to $\Tn$, and denote it as
$h=h_0+t^2h_2+t^3h_3+O(t^4)$ from \eq{hst3eq1}, where $h_0$ is the
induced metric of $g_0$ on $\Tn$ and consider
$h,\,h_0,\,h_2,\,h_3$  as matrices. A direct computation gives \ea
\det(h)&=\det \bigl(h_0+t^2h_2+t^3h_3+O(t^4)\bigr)
\nonumber\\&=\det(h_0)\det
\bigl(I+t^2h_0^{-1}h_2+t^3h_0^{-1}h_{3}+O(t^4)\bigr),\label{hst5eq2}\ea
and $$\sqrt{\det(h)}= \sqrt{\det(h_0)}\bigl(1+\f{1}{2}t^2
\mathop{\rm Tr}(h_0^{-1}h_2)+\f{1}{2}t^3\mathop{\rm
Tr}(h_0^{-1}h_{3})+O(t^4)\bigr),$$ where $h_0^{-1}$ is the inverse
matrix of $ h_0$. We thus have \ea
K^t(p,\up)&=F^t_{p,\up}(f^t_{p,\up})=\Vol_{g^t_{p,\up}}\Phi
\bigl(\Ga_{\d f^t_{p,\up}}\bigr)\nonumber\\&=
\int_{\Tn}\bigl(1+\f{1}{2}t^2 \mathop{\rm
Tr}(h_0^{-1}h_2)+\f{1}{2}t^3\mathop{\rm
Tr}(h_0^{-1}h_{3})+O(t^4)\bigr)\,\d
V_0\nonumber\\&=F_0(f^t_{p,\up})+F^{t}_{p,\up}(0)-F_0(0)+O(t^4).\label{hst5eq3}\ea
In the last equality, we use the observation that from
\eq{hst3eq1}  and \eq{apeq4}  one can totally separate the
contribution of $f^t_{p,\up}$ and $g^t_{p,\up}$ to $h_2$ and
$h_3$, and then do the same expansion for $F_0(f^t_{p,\up})$ and
$F^{t}_{p,\up}(0)$. On the other hand, we have \e
F_0(f^t_{p,\up})=F_0(0)+\f{\d}{\d
s}F_0(sf^t_{p,\up})\big\vert_{s=0}+O(t^4)=F_0(0)+O(t^4),\label{hst5eq4}\e
where we have used \eq{apeq4} in the first equality and $\Tn$ is
Hamiltonian stationary w.r.t. $g_0$ in the second one. Plugging
\eq{hst5eq4} into \eq{hst5eq3}, we get
$K^t(p,\up)=F^{t}_{p,\up}(0)+O(t^4)$ as desired.
\end{proof}
\begin{rem}
Proposition~\ref{hst5prop1} in fact not only works  for $\Tn$, but
also for all compact Hamiltonian stationary Lagrangians $L$ in
$\C^n$. To identify $dK^t$ with $H^t$ we need $L$ to be
Hamiltonian rigid.
\end{rem}
\begin{prop}
Further expansion gives
$$K^t(p,\up)=\bigl(1-\f{1}{4}t^2\sum_i^n
a_i^2R_{i\bar{i}i\bar{i}}(p)\bigr)\Vol_{g_0}(\Tn)+O(t^4).$$
\label{hst5prop2}\end{prop}
\begin{proof}
 The induced
metric $h^t_{p,\up}$ of $g^t_{p,\up}$ on $\Tn $ is given in
\eq{hst2eq4} and the inverse matrix $h_0^{-1}$
 has entries $h_0^{ij}=\f{\de _{ij}}{a_i^2}$. Combining \eq{hst2eq4} and
a similar discussion as in \eq{hst5eq3}, we get \e
F^{t}_{p,\up}(0)=F_0(0)-\f{1}{2}\int_{\Tn}\sum_{i}^n\bigl(t^2\Re
A_{ii}+t^3\Re C_{ii}\bigr)\,\d V_0+O(t^4),\label{hst5eq5}\e where
$\Re A_{ii}$ and $\Re C_{ii}$ are as in \eq{hst2eq3}. Noting that
the integration of $\cos$ and $\sin$ function on $\Tn$ will
vanish, only the terms involving $R_{i\bar{i}i\bar{i}}(p)$ remain.
We thus have $F^{t}_{p,\up}(0)=\bigl(1-\f{1}{4}t^2\sum_i^n
a_i^2R_{i\bar{i}i\bar{i}}(p)\bigr)F_0(0)+O(t^4).$ Combining with
Proposition~\ref{hst5prop1}, it follows that \e
K^t(p,\up)=\bigl(1-\f{1}{4}t^2\sum_i^n
a_i^2R_{i\bar{i}i\bar{i}}(p)\bigr)F_0(0)+O(t^4).\label{hst5eq6}\e
This complete the proof by replacing $F_0(0)$ by
$\Vol_{g_0}(\Tn)$.
\end{proof}
\begin{proof}[Proof of Theorem~\ref{hst4thm2}]
Noting that both $K^t$ and $\sum_i^n a_i^2R_{i\bar{i}i\bar{i}}(p)$
are invariant under $T^n$ action, we now {\bf consider $K^t$ as a
map from $U/T^n$ to $\R$}. We have $K^0(p,[\up])\equiv F_0(0)$ and
every $(p,[\up])$ is a critical point of $K^0$. Suppose that
$(p_0,[\up_0])$ is a non-degenerate critical point of
$F_{a_1,\ldots,a_n}(p,[\up])=\sum_i^n
a_i^2R_{i\bar{i}i\bar{i}}(p)$. We particular look $K^0(p,[\up])$
at $(p_0,[\up_0])$. It follows from the Implicit Function Theorem
that there exists a smooth function $(p\,(t),[\up(t)])\in U/T^n$
for small $t$ with $(p\,(0),[\up(0)])=(p_0,[\up_0)]$, such that
$(p\,(t),[\up(t)])$ is a critical point of $K^t$. Moreover,  the
distance between $(p\,(t),[\up(t)])$ and $(p_0,[\up_0)]$ in
$U/T^n$ is  bounded by $ct^2$ from \eq{hst5eq6}. Therefore, the
Hamiltomian stationary Lagrangian tori with radii $(ta_1,\ldots,
ta_n)$ obtained are $T^n$ invariant and will not intersect each
other for small $t$. This completes the proof of the theorem.
\end{proof}
\appendix \section*{Appendix}
Here we give a supplement to \eq{hslteq5}. Rewrite $\d^*\al_{H_f}$
in \eq{hslt4eq1} as $\d^*\al(t,p,\up,f)$ to indicate its
dependency.  We have \begin{equation}
\begin{aligned} &\;\d^*\al(t,p,\up,f)\\=&\;\d^* \al ,
(t,p,\up,0)+\int_0^1 \f{d}{ds}\d^* \al (t,p,\up,sf) \,ds\\=&\;\d^*
\al (t,p,\up,0)+\int_0^1 (\d^* \al)_{\psi} (t,p,\up,sf)f
\,ds\\=&\;\d^* \al (t,p,\up,0)+\int_0^1 \bigl(-{\cal
\bar{L}}^t_{p,\up}f+\int_0^1 \f{d}{du}\bigl((\d^* \al)_{\psi}
(t,p,\up,usf)f\bigr)\,du \bigr)\,ds\\=&\;\d^* \al
(t,p,\up,0)-{\cal \bar{L}}^t_{p,\up}f+\int_0^1 \int_0^1 (\d^*
\al)_{\psi \psi} (t,p,\up,usf)sf^2\,du \,ds, \label{apeq1}
\end{aligned}
\end{equation}
where  ${\cal \bar{L}}^t_{p,\up}f= -(\d^* \al)_{\psi}
(t,p,\up,0)f$ is as \eq{hst21eq6} computed w.r.t. $g^t_{p,\up}$.
Denote \e{\cal Q}^t_{p,\up}(f)=\int_0^1 \int_0^1 (\d^* \al)_{\psi
\psi} (t,p,\up,usf)sf^2\,du \,ds.\label{apeq2}\e The induced
metric of
 $g_{p,\up}^t$ on $\Tn$ is uniformly bounded from
 Proposition~\ref{hs3prop2},
 and so is $\d^* \al (t,p,\up,0)$. If the norm of $f$ is small, it will not
 change the induced metric too much either. More precisely,  we
  have $\bnm {\d^* \al (t,p,\up,f)}_{k,\ga}$ bounded if $\nm f
  _{k+4,\ga}$ is small. Similarly, we have
  $\bnm {(\d^* \al)_{\psi \psi} (t,p,\up,usf)}_{k,\ga}$ bounded for $0\le s\le 1$ and $0\le u\le 1$.
  Thus \eq{apeq2} gives
  \begin{equation}
  \bnm {{\cal Q}^t_{p,\up}(f)}_{k,\ga}\le c\nm {f}_{k+4,\ga}^2.
  \label{apeq3}
  \end{equation}
By Lemma~\ref{hst2lem1}, we have $\bnm{\d^* \al
(t,p,\up,0)}{}_{C^{k,\ga}}=O(t^2$) for any integer $k\ge 0$ and
$\ga\in(0,1)$. It leads to \e \nm{f^t_{p,\up}}{}_{C^{k+4,\ga}}\le
ct^2 \label{apeq4}\e from the Implicit Function Theory in the
construction in Theorem~\ref{hs5thm}. And thus
 $\d^*\al(t,p,\up,f^t_{p,\up})=O(t^2)$ from  \eq{apeq1}, \eq{hst21eq6},
 \eq{apeq3} and \eq{apeq4}. In \eq{hslt4eq1}, we have $J^{t,f}_{p,\up} =1+O(t^2)$
 from \eq{hst2eq2} and \eq{apeq4}. Therefore, \eq{hslt4eq1} gives
$$P^t_{p,\up}(f^t_{p,\up})=-\d^*\al(t,p,\up,f^t_{p,\up})+O(t^4)
=-P\d^*\al(t,p,\up,f^t_{p,\up})+O(t^4)$$ since $P^t_{p,\up}
(f^t_{p,\up})$ is indeed in $\Ker{\cal L}$ by
Theorem~\ref{hs5thm}.

Combining \eq{apeq3} and \eq{apeq4}, it gives ${\cal
Q}^t_{p,\up}(f^t_{p,\up})=O(t^4)$. Noting  that ${\cal
\bar{L}}^0_{p,\up}={\cal L}$, by Proposition~\ref{hs3prop2} and
\eq{hst21eq6} we have \e \bnm{{\cal \bar{L}}^t_{p,\up}(f)-{\cal
L}(f)}{}_{C^{k,\ga}}\le ct^2 \nm{f}_{C^{k+4,\ga}}.\label{apeq5}\e
Since $\cal L$ is a self-adjoint operator,
 ${\cal
L}(f^t_{p,\up})$ is perpendicular to $\Ker{\cal L}$. We then have
$P{\cal \bar{L}}^t_{p,\up}(f^t_{p,\up})=O(t^4)$ by \eq{apeq4} and
\eq{apeq5}. Therefore, $$P^t_{p,\up} (f^t_{p,\up})=-P\,
\d^*\al(t,p,\up,0) +O(t^4),$$ which is \eq{hslteq5} by denoting
$\d^*\al(t,p,\up,0)$ as $\d^*\al^t_{p,\up}$.


\begin{thebibliography}{99}
\bibitem{AmOh} A. Amarzaya and Y. Ohnita,  {\it `Hamiltonian
stability of parallel Lagrangian submanifolds in complex space
forms'}, preprint.

\bibitem{BuCo} A. Butscher and J. Corvino, {\it `Hamiltonian
stationary tori in K\"ahler manifolds'}, arXiv:0811.2829, 2008.

\bibitem{HR1} F. H\'{e}lein and P. Romon, {\it `Hamiltonian
stationary surfaces in $\C^2$'}, Comm. Anal. Geom. 10 (2002), no.
1, 79--126.

\bibitem{HR2} F. H\'{e}lein and P. Romon, {\it `Hamiltonian
stationary tori in the complex projective plane'}, Proc. London.
Math. Soc. (3) 90 (2005), no. 2, 472--496.

\bibitem{Joyc2} D. Joyce, {\it `Special Lagrangian submanifolds with
isolated conical singularities. I. Regularity'}, Ann. Global Anal.
Geom. 25 (2004), 201--251. math.DG/0211294.

\bibitem{JLS}  D. Joyce, Y.-I Lee and R. Schoen
{\it `On the existence of Hamiltonian stationary Lagrangian
submanifolds in symplectic manifolds'}, \hfil\break
arXiv:0902.3338, 2009.




\bibitem{McSa} D. McDuff and D. Salamon, {\it `Introduction to
Symplectic Topology'}, second edition, OUP, Oxford, 1998.

\bibitem{Mose} J.K. Moser, {\it `On the volume elements on
manifolds'}, Trans. Amer. Math. Soc. 120 (1965), 280--296.

\bibitem{Oh1} Y.-G. Oh, {\it `Second variation and stabilities of
minimal Lagrangian submanifolds in K\"ahler manifolds'}, Invent.
math. 101 (1990), 501--519.

\bibitem{Oh2} Y.-G. Oh, {\it `Volume minimization of Lagrangian
submanifolds under Hamiltonian deformation'}, Math. Z. 212 (1993),
175--192.

\bibitem{Tian} G. Tian, {\it `Canonical Metrics in K\"ahler Geometry'},
 Lectures in Mathematics ETH Zurich. Birkhauser Verlag, Basel,
2000.

\bibitem{Ye} R. Ye, {\it `Foliation by Constant Mean Curvature Spheres'},
Pacific J. 147 (1991), no. 2, 381--396.

\end{thebibliography}
\end{document}